\documentclass[reqno]{amsart}

\usepackage{style}

\author{Mireille Boutin}
\address{Department of Mathematics, Purdue
  University, 150 N. University St., West Lafayette, IN, 47906, USA}
\email{mboutin@purdue.edu}
\author{Gregor Kemper} \address{Technische Universit\"at M\"unchen,
  Zentrum Mathematik - M11, Boltzmannstr. 3, 85748 Garching, Germany}
\email{kemper@ma.tum.de}

\title{Can a Ground-Based Vehicle Hear the Shape of a Room?}

\date{\today}


\begin{document}

\begin{abstract}
  Assume that a ground-based vehicle moves in a room with walls or other planar surfaces. Can the vehicle reconstruct the positions of the walls from the echoes of a single sound event? We assume that the vehicle carries some microphones and that a loudspeaker is either also mounted on the vehicle or placed at a fixed location in the room. We prove that the reconstruction is almost always possible if (1) no echoes are received from floors, ceilings or sloping walls and the vehicle carries at least three non-collinear microphones, or if (2) walls of any inclination may occur, the loudspeaker is fixed in the room and there are four non-coplanar microphones.
  
  The difficulty lies in the echo-matching problem: how to determine which echoes come from the same wall. We solve this by using a Cayley-Menger determinant. Our proofs use methods from computational commutative algebra.
\end{abstract}

\maketitle

\section*{Introduction} \label{sIntro}%

This paper is concerned with the problem of reconstructing the position of walls and other planar surfaces using echoes. More specifically, an omni-directional loudspeaker produces a short, high frequency impulse. The echoes of the impulses are captured by some microphones.
The loudspeaker and the microphones are assumed to be synchronized.
The walls are modeled using {\em mirror points}, which are the reflections of the loudspeaker with respect to the planar surfaces. 
The times of arrival of an echo give us the total distance travelled by the sound, 
which is equal to the distance from the microphone to a mirror point. 
We are interested in determining to what extent it is theoretically possible to reconstruct the wall positions from the distance measurements obtained in this fashion. All measurements are assumed to be exact, and the computations are assumed to be performed with infinite precision. Yet, even in this theoretical scenario, it may be impossible, at least in some circumstances, to correctly reconstruct the wall positions. One of the goal of our work is to narrow down the problematic cases and show that one can expect to avoid them in some common application scenarios. Here, we are thinking of a vehicle carrying the microphones and having the ability to move more or less freely on the ground, for example, a robot rolling on the floor inside a warehouse.

In previous work~[\citenumber{Boutin:Kemper:2019}], we considered the case of a drone or, more generally, a vehicle with all six degrees of freedom (the three coordinate axes and yaw, pitch and roll). If there are four microphones mounted on the vehicle and there is a loudspeaker either on the vehicle or at a fixed location in the room, then our algorithm~[\citenumber{Boutin:Kemper:2019}] reconstructs the position of every wall from which an echo is received by all microphones. But it cannot be avoided that for certain ``bad'' vehicle positions a ``ghost wall'', one which is not really there, is detected. Our main result states that the bad positions are rare. However, to get out of a possibly bad position the vehicle may need to use all six degrees of freedom, including pitch and roll. For a drone this means that being in a good position may result in receiving a thrust in some direction, so the drone cannot hover while carrying out the wall reconstruction. So we are interested in investigating whether our results carry over to a situation where the degrees of freedom are restricted to the coordinate axes and yaw. Perhaps more interesting is the case of a ground-based vehicle. These are the scenarios considered in this paper.

The restriction of degrees of freedom not only makes the mathematical arguments harder but also yields somewhat weaker results. In fact, we show in \cref{sStacks} that there are ``unlucky'' wall arrangements for which it is not true that almost all vehicle positions are good, if the degrees of freedom are restricted as stated above. So the cases of a ``hovering drone'' and a ground-based vehicle really are more difficult. The first main result (\cref{2tMain}) deals with the case of a ground-based vehicle and a loudspeaker placed at a fixed position in the room, and gives a precise characterization of the wall arrangements for which it is not true that almost all vehicle positions are good. These are the arrangements where an ``unlucky stack of mirror points'', as defined in \cref{sStacks}, occurs, and they are themselves very rare. The second main result (\cref{2tHovering}) treats the case of a hovering drone (four degrees of freedom) and, perhaps surprisingly, reaches the same conclusion: the additional degree of freedom does not result in fewer exceptional wall arrangements. The proofs of these two results are similar in spirit to the proofs in~[\citenumber{Boutin:Kemper:2019}], and in particular involve large computations in ideal theory done by computer. But apart from the new aspect of exceptional wall arrangements, there arise further theoretical difficulties that are explained after the statement of \cref{2tMain}. Analyzing our proof method shows that for each degree of freedom taken away, the final computation requires roughly two additional variables. This is why computational bottlenecks prevented us from achieving general results in the case that the loudspeaker is mounted on the vehicle.

In the first result of this paper, the vehicle moves in two dimensions, but still lives in three-dimensional space, meaning that the walls it detects may be floors, ceilings, and sloping walls. The last result (\cref{3t2D}) of this paper concerns the truly two-dimensional case. In real-world terms, this means that no echoes should be received from floors, ceilings, and sloping walls, since they cannot be dealt with. This case is much easier, has no exceptional wall arrangements, only requires three microphones, and also allows for the loudspeaker to be mounted on the vehicle, which we were not able to address in a satisfactory way in three dimensions.

In this paper, a room is just a spacial arrangement of walls, and a wall is just a plane surface. We only consider first order echoes, meaning that the sound has not bounced twice or more often. A further assumption is that the sound frequency is high enough to justify the use of ray acoustics. When we say that almost all vehicle positions are good, we mean that within the configuration space of all vehicle positions, the bad ones form a subset of volume zero. In fact, we use this language only if the bad positions are contained in a subvariety of lower dimension. Intuitively ``almost all'' can be thought of as ``with probability one.''

\par{\bf Acknowledgments.} This work has benefited immensely from a research stay of the two authors at the Banff International Research Station for Mathematical Innovation and Discovery (BIRS) under the ``Research in Teams'' program. We would like to thank BIRS for its hospitality and for providing an optimal working environment.

\section{Related work}
The relationship between the geometry of a room and the source-to-receiver acoustic impulse response has long been a subject of interest, in particular for the design of concert halls with desirable acoustic properties. For example, the two-point impulse response within rectangular rooms was analyzed in [\citenumber{alien1976image}]. The results were later extended to arbitrarily polyhedral rooms  in [\citenumber{borish1984extension}]. Subsequent work used the impulse response of a room to reconstruct its geometry, including that of objects within the room (e.g. [\citenumber{moebus2007three}]).

Some reconstruction methods are focused on 2D room geometry (e.g., [\citenumber{dokmanic2011can}],\citenumber{antonacci2012inference}], [\citenumber{el2016reflector}]) while others consider a more general but still constrained 3D geometry. For example,  \mycite{ParkChoi2021} reconstruct the wall surface for a convex, polyhedral and bounded room using several loudspeakers and microphones in known positions. Other methods are applicable to reconstructing of arbitrarily planar surfaces in 3D (e.g. [\citenumber{DPWLV1}]).

Some methods reconstruct wall points directly  (e.g., [\citenumber{canclini2011exact}], [\citenumber{filos2011robust}], [\citenumber{remaggi20153d}], [\citenumber{el2016reflector}], [\citenumber{ParkChoi2021}])
while others, like us, reconstruct mirror points representing the reflection of a loudspeaker with respect to a wall  plane (e.g.,  [\citenumber{tervo20123d}], [\citenumber{mabande2013room}], [\citenumber{DPWLV1}], [\citenumber{remaggi2016acoustic}], [\citenumber{Baba2017}]).
Some, like us, assume that the microphones are synchronized (e.g.,  [\citenumber{tervo20123d}],  [\citenumber{DPWLV1}],[\citenumber{jager2016room}],  [\citenumber{rajapaksha2016geometrical}],), while others do not (e.g., [\citenumber{Scheuing2008}], [\citenumber{pollefeys2008direct}],  [\citenumber{antonacci2012inference}]).

In certain setups,  the microphones are placed on a vehicle such as a robot (e.g. [\citenumber{peng2015room}]) or a drone [\citenumber{Boutin:Kemper:2019}]. Recent work  uses cell phones as a cheaper alternative. For example, \mycite{shih2019can} place an omnidirectional speaker at a known position in a polyhedral room 
and use a cell phone to record the echoes while \mycite{zhou2017batmapper} assume a smart phone is  held horizontally by a person walking along rooms and corridors. 

The vast majority of the work in the literature is focused on the numerical reconstruction 
of the room, including the difficult problem of labeling echoes coming from the same surface. However, from a theoretical standpoint, there is still a lot to be understood regarding the well-posedness of different reconstruction setups, especially for non-generic room geometries or receiver and source placements.

\section{Unlucky stacks of mirror points} \label{sStacks}

\newcommand{\ve}[1]{\mathbf{#1}}%
\newcommand{\re}{\operatorname{ref}}%
\newcommand{\ini}{{\operatorname{ini}}}%
\newcommand{\Mini}{M_\ini}%
\newcommand{\ASO}{\operatorname{ASO}}%
\newcommand{\h}{{\operatorname{hom}}}
\newcommand{\vertex}[1][black]{\node[#1,circle, draw, outer sep=3pt,
  inner sep=0pt, minimum size=3pt,fill]}

As mentioned above, the primary
challenge in wall detection from echoes is the matching of echoes
heard by different microphones, i.e., determining which echoes come
from the same wall and which ones do not.


Let us recall an example from~[\citenumber{Boutin:Kemper:2019}] where
this matching inevitably goes wrong, leading to a ghost wall. As it
turns out, this type of example will play a crucial role in the present paper.
\cref{1fBad}%
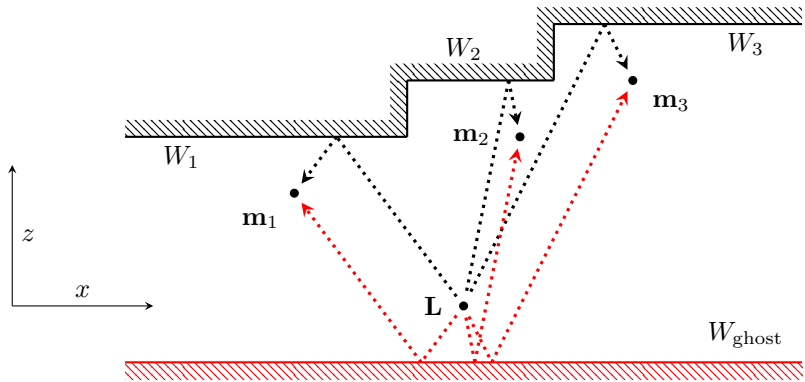
\begin{figure}[htbp]
  \centering
  \begin{tikzpicture}[scale=0.75]
    \fill[pattern=north west lines,pattern color=red] (1,1) --
    (13,1) -- (13,0.7) -- (1,0.7) -- (1,1);
    \draw[red,thick] (1,1) -- (13,1);
    \draw (12,1.1) node[above]{$W_{\operatorname{ghost}}$};
    \vertex (L) at (7,2)  [label=left:$\ve L$] {};
    \vertex (m1) at (4,4) [label=-120:$\ve m_1$] {};
    \vertex (m2) at (8,5) [label=left:$\ve m_2\ $] {};
    \vertex (m3) at (10,6) [label=-30:$\ve m_3$] {};
    \fill[pattern=north west lines] (1,5) -- (6,5) -- (6,6) --
    (8.6,6) -- (8.6,7) -- (13,7) -- (13,7.3) -- (8.3,7.3) --
    (8.3,6.3) -- (5.7,6.3) -- (5.7,5.3) -- (1,5.3) -- (1,5);
    \draw[thick] (1,5) -- (6,5);
    \draw (2,5) node[below]{$W_1$};
    \draw[thick] (6,5) -- (6,6);
    \draw[thick] (6,6) -- (8.6,6);
    \draw (7,6.2) node[above]{$W_2$};
    \draw[thick] (8.6,6) -- (8.6,7);
    \draw[thick] (8.6,7) -- (13,7);
    \draw (12,7) node[below]{$W_3$};
    \draw[very thick,dotted,>=stealth,->] (3/4*4+7/4,5) -- (m1);
    \draw[very thick,dotted] (L) -- (3/4*4+7/4,5);
    \draw[very thick,dotted,>=stealth,->] (4/5*8+1/5*7,6) -- (m2);
    \draw[very thick,dotted] (4/5*8+1/5*7,6) -- (L);
    \draw[very thick,dotted,>=stealth,->] (5/6*10+7/6,7) -- (m3);
    \draw[very thick,dotted] (5/6*10+7/6,7) -- (L);
    \draw[very thick,dotted,red,>=stealth,->] (3/4*7+4/4,1) -- (m1);
    \draw[very thick,dotted,red] (L) -- (3/4*7+4/4,1);
    \draw[very thick,dotted,red,>=stealth,->] (4/5*7+8/5,1) -- (m2);
    \draw[very thick,dotted,red] (L) -- (4/5*7+8/5,1);
    \draw[very thick,dotted,red,>=stealth,->] (5/6*7+10/6,1) -- (m3);
    \draw[very thick,dotted,red] (L) -- (5/6*7+10/6,1);
    \draw[>=stealth,->] (-1,2) -- (1.5,2) node[midway,above]{$x$};
    \draw[>=stealth,->] (-1,2) -- (-1,4.5) node[midway,right]{$z$};
  \end{tikzpicture}
  \caption{The microphones $\ve m_i$ think they are hearing echoes
    from the wall $W_{\operatorname{ghost}}$. The dotted lines stand
    for sound rays, $\ve L$ for the loudspeaker.}
  \label{1fBad}
\end{figure}
shows three microphones in a plane at positions~$\ve m_1$, $\ve m_2$,
$\ve m_3$ that hear echoes from three walls~$W_i$, but the time
elapsed between sound emission and echo detection is the same as if
they were hearing echoes from one single wall
$W_{\operatorname{ghost}}$, which does not exist. This arises because
(1) the walls $W_i$ are all horizontal, (2) the distances
between $\ve m_i$ and $W_i$ are the same for all~$i$, and (3)
each~$\ve m_i$ can hear the echo from $W_i$. It is easy to add a
fourth microphone outside of the drawing plane, together with a wall
possibly also outside of the plane, such that~(1)--(3) extend to the
fourth microphone and wall. (We find it harder to include that in our
two-dimensional sketch in \cref{1fBad}.) Since the echoes heard by the
microphones behave precisely like echoes from the single wall
$W_{\operatorname{ghost}}$, any matching procedure will falsely assume
that they have, in fact, come from $W_{\operatorname{ghost}}$, giving
rise to a ghost wall.
The situation is particularly severe since the ghost wall has some
persistence properties. In fact, the conditions~(1)--(3), and
therefore the emergence of a ghost wall, will be preserved if
\begin{enumerate}[label=(\alph*)]
\item \label{PerA} the microphones are moved, independently and within a certain
  range, in directions parallel to the walls,
\item \label{PerB} the microphones are moved, {\em together} and within a certain
  range, along the $z$-axis, or, less importantly,
\item \label{PerC} the loudspeaker is moved, within a certain range, in any
  direction.
\end{enumerate}

Our sketch does not contain any vehicle that carries the
microphones. Let us now imagine there is such a vehicle, and that it
is {\em ground-based}, in the sense that its movement (rotation and
translation) is restricted to the $x$-$y$-plane, where the $y$-axis is
perpendicular to the drawing plane. The vehicle may also be allowed to
move up or down along the $z$-axis, like a hovering drone. Then the
persistence properties imply that within a certain range, vehicle
movements will not make the ghost wall disappear. For this it does not
matter whether the loudspeaker is mounted on the vehicle or at a fixed
position. In all cases, there is a region of positive volume within
the space of all vehicle movements, in which the vehicle remains in a
bad position. So it is not true that almost all vehicle positions are good. Thus
here we hit a limitation to what can possibly be shown for microphones
mounted on ground-based vehicles, or vehicles with an additional
degree of freedom of moving up or down. (In comparison, a drone has
two additional degrees of freedom: pitch and roll.)



The following closer analysis will reveal a more general situation where it is not true that almost all positions are good. Since we are using ray acoustics, an echo reflected at a wall arrives
with the same time delay as if it were emitted at what we call the
\df{mirror point}, which is the reflection of the loudspeaker position
at the wall or, more precisely, at the plane containing the wall
surface. \cref{1fMirror} shows the mirror points~$\ve s_i$ in the situation of
\cref{1fBad}.

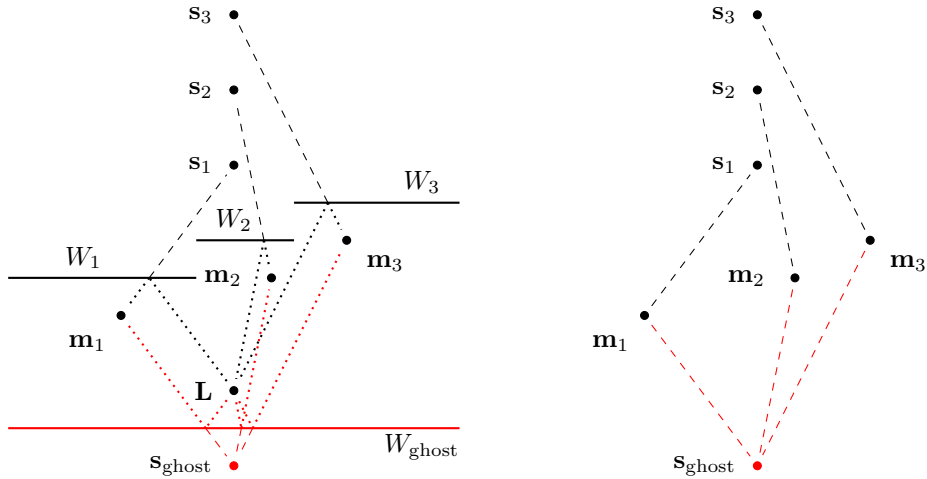
\begin{figure}[htbp]
  \centering
  \begin{tikzpicture}[scale=0.5]
    \draw[red,thick] (1,1) -- (13,1);
    \draw (12,1.1) node[below]{$W_{\operatorname{ghost}}$};
    \vertex (L) at (7,2)  [label=left:$\ve L$] {};
    \vertex (s1) at (7,8) [label=left:$\ve s_1$] {};
    \vertex (s2) at (7,10) [label=left:$\ve s_2$] {};
    \vertex (s3) at (7,12) [label=left:$\ve s_3$] {};
    \vertex (sg) at (7,0) [color=red,label=left:$\ve s_{\operatorname{ghost}}$] {};
    \vertex (m1) at (4,4) [label=-120:$\ve m_1$] {};
    \vertex (m2) at (8,5) [label=left:$\ve m_2\ $] {};
    \vertex (m3) at (10,6) [label=-30:$\ve m_3$] {};
    \draw[thick] (1,5) -- (6,5);
    \draw (3,5) node[above]{$W_1$};
    \draw[thick] (6,6) -- (8.6,6);
    \draw (7,6) node[above]{$W_2$};
    \draw[thick] (8.6,7) -- (13,7);
    \draw (12,7) node[above]{$W_3$};
    \draw[thick,dotted] (3/4*4+7/4,5) -- (m1);
    \draw[dashed] (3/4*4+7/4,5) -- (s1);
    \draw[thick,dotted] (L) -- (3/4*4+7/4,5);
    \draw[thick,dotted] (4/5*8+1/5*7,6) -- (m2);
    \draw[dashed] (4/5*8+1/5*7,6) -- (s2);
    \draw[thick,dotted] (4/5*8+1/5*7,6) -- (L);
    \draw[thick,dotted] (5/6*10+7/6,7) -- (m3);
    \draw[dashed] (5/6*10+7/6,7) -- (s3);
    \draw[thick,dotted] (5/6*10+7/6,7) -- (L);
    \draw[thick,dotted,red] (3/4*7+4/4,1) -- (m1);
    \draw[dashed,red] (3/4*7+4/4,1) -- (sg);
    \draw[thick,dotted,red] (L) -- (3/4*7+4/4,1);
    \draw[thick,dotted,red] (4/5*7+8/5,1) -- (m2);
    \draw[dashed,red] (4/5*7+8/5,1) -- (sg);
    \draw[thick,dotted,red] (L) -- (4/5*7+8/5,1);
    \draw[thick,dotted,red] (5/6*7+10/6,1) -- (m3);
    \draw[dashed,red] (5/6*7+10/6,1) -- (sg);
    \draw[thick,dotted,red] (L) -- (5/6*7+10/6,1);
  \end{tikzpicture} \qquad \qquad
  \begin{tikzpicture}[scale=0.5]
    \vertex (s1) at (7,8) [label=left:$\ve s_1$] {};
    \vertex (s2) at (7,10) [label=left:$\ve s_2$] {};
    \vertex (s3) at (7,12) [label=left:$\ve s_3$] {};
    \vertex (sg) at (7,0) [color=red,label=left:$\ve s_{\operatorname{ghost}}$] {};
    \vertex (m1) at (4,4) [label=-120:$\ve m_1$] {};
    \vertex (m2) at (8,5) [label=left:$\ve m_2\ $] {};
    \vertex (m3) at (10,6) [label=-30:$\ve m_3$] {};
    \draw[dashed] (m1) -- (s1);
    \draw[dashed] (m2) -- (s2);
    \draw[dashed] (m3) -- (s3);
    \draw[dashed,red] (m1) -- (sg);
    \draw[dashed,red] (m2) -- (sg);
    \draw[dashed,red] (m3) -- (sg);
  \end{tikzpicture}
  \caption{Virtually, the sound comes from the mirror points
    $\ve s_i$. So the loudspeaker and walls might as well be left out
    of the sketch.}
  \label{1fMirror}
\end{figure}

\begin{figure}[htbp]
  \centering
  \begin{tikzpicture}[scale=0.5]
    \draw[red,thick] (6.5-0.5*3.5,1.75-0.5) -- (6.5+0.8*3.5,1.75+0.8);
    \draw (10,2.3) node[below]{$W_{\operatorname{ghost}}$};
    \vertex (L) at (6,3.5)  [label=left:$\ve L$] {};
    \vertex (s1) at (7,8) [label=left:$\ve s_1$] {};
    \vertex (s2) at (7,10) [label=left:$\ve s_2$] {};
    \vertex (s3) at (7,12) [label=left:$\ve s_3$] {};
    \vertex (sg) at (7,0) [color=red,label=left:$\ve s_{\operatorname{ghost}}$] {};
    \vertex (m1) at (4,4) [label=-120:$\ve m_1$] {};
    \vertex (m2) at (8,5) [label=right:$\ve m_2\ $] {};
    \vertex (m3) at (10,6) [label=-30:$\ve m_3$] {};
    \draw[thick] (6.5-0.5*4.5,5.75+0.5) -- (6.5-0.1*4.5,5.75+0.1*1);
    \draw (3.9,5.2) node[above]{$W_1$};
    \draw[thick] (6.5-0*6.5,6.75) -- (6.5+0.25*6.5,6.75-0.25);
    \draw (6.8,5.7) node[above]{$W_2$};
    \draw[thick] (6.5+0.25*8.5,7.75-0.25) -- (6.5+0.5*8.5,7.75-0.5);
    \draw (11,7.1) node[above]{$W_3$};
    \draw[thick,dotted] (4+0.494*3,4+0.494*4) -- (m1);
    \draw[dashed] (4+0.494*3,4+0.494*4) -- (s1);
    \draw[thick,dotted] (L) -- (4+0.494*3,4+0.494*4);
    \draw[thick,dotted] (8-0.3135,5+0.3135*5) -- (m2);
    \draw[dashed] (8-0.3135,5+0.3135*5) -- (s2);
    \draw[thick,dotted] (8-0.3135,5+0.3135*5) -- (L);
    \draw[thick,dotted] (10-0.237*3,6+6*0.237) -- (m3);
    \draw[dashed] (10-0.237*3,6+6*0.237) -- (s3);
    \draw[thick,dotted] (10-0.237*3,6+6*0.237) -- (L);
    \draw[thick,dotted,red] (4+3*0.61,4-4*0.61) -- (m1);
    \draw[dashed,red] (4+3*0.61,4-4*0.61) -- (sg);
    \draw[thick,dotted,red] (L) -- (4+3*0.61,4-4*0.61);
    \draw[thick,dotted,red] (8-0.6,5-5*0.6) -- (m2);
    \draw[dashed,red] (8-0.6,5-5*0.6) -- (sg);
    \draw[thick,dotted,red] (L) -- (8-0.6,5-5*0.6);
    \draw[thick,dotted,red] (10-3*0.632,6-6*0.632) -- (m3);
    \draw[dashed,red] (10-3*0.632,6-6*0.632) -- (sg);
    \draw[thick,dotted,red] (L) -- (10-3*0.632,6-6*0.632);
  \end{tikzpicture} \qquad \qquad
  \begin{tikzpicture}[scale=0.5]
    \vertex (s1) at (7,8) [label=left:$\ve s_1$] {};
    \vertex (s2) at (7,10) [label=left:$\ve s_2$] {};
    \vertex (s3) at (7,12) [label=left:$\ve s_3$] {};
    \vertex (sg) at (7,0) [color=red,label=left:$\ve s_{\operatorname{ghost}}$] {};
    \vertex (m1) at (4,4) [label=-120:$\ve m_1$] {};
    \vertex (m2) at (8,5) [label=left:$\ve m_2\ $] {};
    \vertex (m3) at (10,6) [label=-30:$\ve m_3$] {};
    \draw[dashed] (m1) -- (s1);
    \draw[dashed] (m2) -- (s2);
    \draw[dashed] (m3) -- (s3);
    \draw[dashed,red] (m1) -- (sg);
    \draw[dashed,red] (m2) -- (sg);
    \draw[dashed,red] (m3) -- (sg);
  \end{tikzpicture}
  \caption{A different loudspeaker position and different walls produce the exact same unlucky stack of mirror points as in \cref{1fMirror}. But here the walls are not horizontal.}
  \label{1fMirror2}
\end{figure}
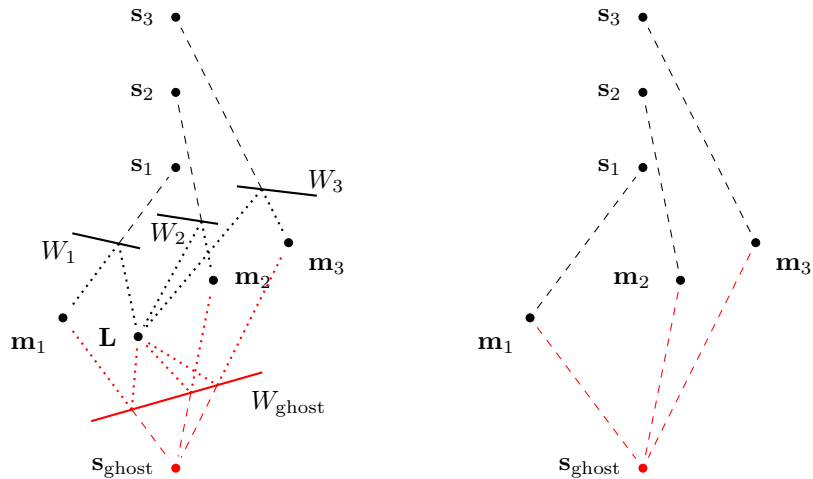

If we have mirror points as in the right sketch of \cref{1fMirror}, then we have the persistence
properties~\ref{PerA} and~\ref{PerB} of the ghost wall discussed above. For this, it is not
necessary that the loudspeaker be situated on the same vertical line
as the mirror points, so the walls may be non-horizontal. \cref{1fMirror2} illustrates that. In fact, it
is only if the loudspeaker is moving as well that we need horizontal
walls to achieve the persistence properties of the ghost wall. It is useful to introduce a catchphrase for the situation shown in \cref{1fMirror,1fMirror2}, which is done in the following definition. From now on, we will only consider the case in which the loudspeaker is at a fixed position and not moving with the vehicle.

\begin{defi} \label{Stack}%
  Assume that four microphones are mounted on a vehicle that can move
  within the $x$-$y$-plane, and possibly up and down along the
  $z$-axis. Assume that the vehicle moves in a scene that contains
  several flat walls and a loudspeaker at a fixed position. We say that the microphones are
  \df{deceived by an unlucky stack of mirror points} if among the walls there are four whose
  mirror points $\ve s_1 \upto \ve s_4$ are contained in a common
  vertical line, such that for $i,j \in \{1 \upto 4\}$, the $z$-coordinate of
  the vector $\ve s_j - \ve s_i$ is twice the $z$-coordinate of
  $\ve m_j - \ve m_i$. In particular, this means that if for some $i \ne j$
  the positions~$\ve m_i$ and~$\ve m_j$ share the same $z$-coordinate, then $\ve s_i = \ve s_j$; so for deception by an unlucky stack of mirror points, the mirror points (and the walls) need not be pairwise distinct.
\end{defi}

Assume that the microphones are deceived by an unlucky stack of mirror points $\ve s_1 \upto \ve s_4$. Then the reflection of~$\ve s_i$ at the horizontal plane containing~$\ve m_i$ produces a point~$\ve s_{\operatorname{ghost}}$ that is the same for all~$i$. Therefore the distance between~$\ve m_i$ and~$\ve s_i$ equals the distance between~$\ve m_i$ and~$\ve s_{\operatorname{ghost}}$. So if the vehicle is positioned in such a way that the $i$th microphone can hear the echo of the $i$th wall, then a ghost wall, corresponding to~$\ve s_{\operatorname{ghost}}$, will be detected. This justifies our speaking of ``deception.''

\section{A ground based vehicle and a hovering drone} \label{sMain}

Before stating our main result it is useful to recall how the wall detection
algorithm from~[\citenumber{Boutin:Kemper:2019}] works. The idea is to
use a relation satisfied by the distances travelled by the sound that
is reflected at the same wall, but received by different
microphones. This relation between the squared distances
$d_1 \upto d_4$ is
\begin{equation} \label{1eqRel}%
  f_D(d_1 \upto d_4) = 0,
\end{equation}
where for real numbers $u_1 \upto u_4 \in \RR$, and with $D_{i,j}$ the
squared distance between the $i$th and the $j$th microphone, the
polynomial $f_D$ is defined as the Cayley-Menger determinant
\begin{equation} \label{1eqF}%
  f_D(u_1 \upto u_4) := \det
  \begin{pmatrix}
    0 & u_1 & \cdots & u_4 & 1 \\
    u_1 & D_{1,1} & \cdots & D_{1,4} & 1 \\
    \vdots & \vdots & & \vdots & \vdots \\
    u_4 & D_{4,1} & \cdots & D_{4,4} & 1 \\
    1 & 1 & \cdots & 1 & 0
  \end{pmatrix}.
\end{equation}
With this, we can state the wall detection algorithm as
\cref{1aDetect}.

\begin{algorithm}
  \caption{Detect walls from first-order echoes} \label{1aDetect}
  \mbox{}%
  \begin{description}
  \item[Input] The delay times of the first-order echoes recorded by
    four microphones, and the distances $D_{i,j}$ between the
    microphones.
  \end{description}
  \begin{algorithmic}[1]
    \STATE \label{1aDetect1} For $i = 1 \upto 4$, collect the recorded
    times of the first-order echoes recorded by the $i$th microphone
    in the set $\mathcal T_i$.%
    \STATE \label{1aDetect2} Set
    $\mathcal D_i := \{c^2 (t - t_0)^2 \mid t \in \mathcal T_i\}$
    ($i = 1 \upto 4$), where~$c$ is the speed of sound and~$t_0$ is
    the time of sound emission.%
    \FOR{$(d_1,d_2,d_3,d_4) \in \mathcal D_1 \times \mathcal D_2
      \times \mathcal D_3 \times \mathcal D_4$} \label{1aDetect3}%
    \STATE \label{1aDetect4} With~$f_D$ defined by~\cref{1eqF},
    evaluate $f_D(d_1 \upto d_4)$.%
    \IF{$f_D(d_1 \upto d_4) = 0$}%
    \STATE \label{1aDetect6} Use the geometric methods such as those
    given in~[\citenumber{Boutin:Kemper:2019}, Proposition~1.1] to
    compute the parameters describing the wall corresponding to the
    distances $d_1 \upto d_4$.%
    \STATE \label{1aDetect7} {\bf Output} the parameters of this
    wall.%
    \ENDIF%
    \ENDFOR%
  \end{algorithmic}
\end{algorithm}

There is no need here to be specific about what parameters are
computed to describe a wall detected by the algorithm. \cref{1aDetect}
is guaranteed to detect every wall from which a first order echo is
heard by all four microphones. However, it may happen that the
relation $f_D(d_1 \upto d_4) = 0$ is satisfied by accident even though
the squared distances~$d_i$ come from sound reflections at different
walls. This can deceive the algorithm into outputting walls that do
not actually exist. As in~[\citenumber{Boutin:Kemper:2019}], we call
such walls {\em ghost walls} and we say that the vehicle carrying the
microphones (and possibly the loudspeaker) is in a {\em bad position}
if at least one ghost wall is detected. One instance where this
happens is when the microphones are deceived by an unlucky stack of mirror points, as
discussed in the previous subsection.

\begin{theorem}[A ground-based vehicle in a scene with a fixed loudspeaker] \label{2tMain}%
  Assume that four microphones are mounted on a ground-based vehicle,
  which can move on the ground plane within a three-dimensional scene containing~$n$ walls and
  a loudspeaker at fixed positions. Assume the microphones do not lie on a common plane. If the microphones are not deceived by an unlucky stack of mirror points, according to \cref{Stack}, then almost all vehicle positions are good.

  Moreover, if~$l \in \{2,3,4\}$ is the number of distinct
  $z$-coordinates of the four microphone positions, then within
  the
  $(3n)$-dimensional configuration space of all wall arrangements, the
  ones where an unlucky stack of mirror points occurs are contained in a subvariety of codimension~$3 (l-1)$.

\end{theorem}

Before giving the proof, we make a remark that should explain why some additional difficulties arise compared to our proofs in~[\citenumber{Boutin:Kemper:2019}]. In~[\citenumber{Boutin:Kemper:2019}] we used the notion of a
``very good position,'' which we briefly recall now. Let $\mathcal W$
be the (finite) set of walls in our room. The mirror points are given
by $\ve s = \re_W(\ve L)$, the reflection of the loudspeaker position
at a wall $W \in \mathcal W$.  We call a vehicle position \df{very
  good} if the following holds: for four walls
$W_1 \upto W_4 \in \mathcal W$ the relation
\[
  f_D\bigl(\lVert\re_{W_1}(\ve L) - \ve m_1\rVert^2 \upto
  \lVert\re_{W_4}(\ve L) - \ve m_4\rVert^2\bigr) = 0
\]
is satisfied only if $W_1 = W_2 = W_3 = W_4$. (Of course, the above
expression depends on the vehicle position because the~$\ve m_i$ do.) It is easy to see, and formally proved
in~[\citenumber{Boutin:Kemper:2019}], that a very good position is
good. The utility of this concept lies in the fact that the very good
positions form a Zariski-open set, so for proving ``allmost all''
statements it suffices to show that for every wall arrangement there
exists at least one very good position. This turned out to be possible
in the cases considered in~[\citenumber{Boutin:Kemper:2019}], where
the microphones are mounted on a drone with six degrees of
freedom. However, for microphones on a ground-based vehicle, it may
happen that no very good position exists, but nevertheless all
positions are good. \cref{1fReallyGood} shows a wall arrangement and
microphone configuration (again in dimensions two) where this happens.

\begin{figure}[htbp]
  \centering
  \begin{tikzpicture}[scale=1.0]
    \fill[pattern=north west lines] (-1,-3) -- (4,0) -- (-1,3) --
    (-0.9,3.15) -- (4.3,0) -- (-0.9,-3.15) -- (-1,-3);
    \draw[thick] (-1,-3) -- (4,0);
    \draw (2,-1.5) node[below]{$W_1$};
    \draw[thick] (-1,3) -- (4,0);
    \draw (2,1.5) node[above]{$W_2$};
    \vertex (L) at (0,0) [label=0:$\ve L$] {};
    \vertex (m1) at (1,0) [color=red,label=0:$\ve m_1$] {};
    \vertex (m2) at (-1.5,1) [color=green,label=-180:$\ve m_2$] {};
    \vertex (m3) at (-1.5,-1) [color=blue,label=-180:$\ve m_3$] {};
    \draw[red,very thick,dotted] (L) -- (-1+0.495*5,-3+0.495*3);
    \draw[red,very thick,dotted,>=stealth,->] (-1+0.495*5,-3+0.495*3) -- (m1);
    \draw[green,very thick,dotted] (L) -- (-1+0.365*5,-3+0.365*3);
    \draw[green,very thick,dotted,>=stealth,->] (-1+0.365*5,-3+0.365*3) -- (m2);
    \draw[blue,very thick,dotted] (L) -- (-1+0.255*5,-3+0.255*3);
    \draw[blue,very thick,dotted,>=stealth,->] (-1+0.255*5,-3+0.255*3) -- (m3);
    \draw[red,very thick,dotted] (L) -- (-1+0.495*5,3-0.495*3);
    \draw[red,very thick,dotted,>=stealth,->] (-1+0.495*5,3-0.495*3) -- (m1);
    \draw[green,very thick,dotted] (L) -- (-1+0.255*5,3-0.255*3);
    \draw[green,very thick,dotted,>=stealth,->] (-1+0.255*5,3-0.255*3) -- (m2);
    \draw[blue,very thick,dotted] (L) -- (-1+0.365*5,3-0.365*3);
    \draw[blue,very thick,dotted,>=stealth,->] (-1+0.365*5,3-0.365*3) -- (m3);
    \draw[>=stealth,->] (-5.5,-1.1) -- (-3.5,-1.1) node[midway,above]{$x$};
    \draw[>=stealth,->] (-5.5,-1.1) -- (-5.5,0.9) node[midway,right]{$z$};
  \end{tikzpicture}
  \caption{The echoes from both walls arrive at $\ve m_1$
    simultaneously. No ghost wall is detected.}
  \label{1fReallyGood}
\end{figure}
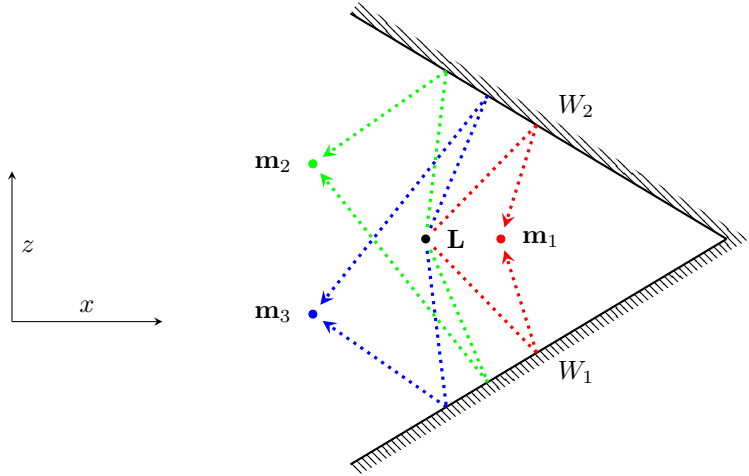

In fact, if the vehicle carrying the microphones and possibly also the
loudspeaker moves in the $x$-$y$-plane, $\ve m_1$ will
remain exactly in the middle between the walls, so $\ve m_1$ receives
the echoes from both walls simultaneously. This implies
\begin{multline*}
  f_D\bigl(\lVert\re_{W_2}(\ve L) - \ve
  m_1\rVert^2,\lVert\re_{W_1}(\ve L) - \ve
  m_2\rVert^2,\lVert\re_{W_1}(\ve L) - \ve
  m_3\rVert^2\bigr) = \\
  f_D\bigl(\lVert\re_{W_1}(\ve L) - \ve
  m_1\rVert^2,\lVert\re_{W_1}(\ve L) - \ve
  m_2\rVert^2,\lVert\re_{W_1}(\ve L) - \ve m_3\rVert^2\bigr) = 0,
\end{multline*}
since in the middle expression all echoes come from the same wall. But
since $W_1 \ne W_2$, this means that the vehicle is {\em not} in a
very good position. Nevertheless, the wall detection algorithm will
not detect a ghost wall in this situation, so the position is good. So
indeed \cref{1fReallyGood} provides an example where no very good
position exists for a ground-based vehicle, but all positions are
good.

We must be able to deal with such situations in our proof. To this
end, we will introduce the notion of a ``really good position,'' which
encapsulates the idea that if the position is not very good, then this
is because of simultaneous arrivals of echoes. The definition of this
notion will be given in the proof of \cref{2tMain}, where we will
also see that a really good position is good. (That very good implies
really good will be obvious.) The problem with really good positions
is that, unlike very good positions, they do not form a Zariski-open
set. So the proof has to take care of this difficulty.



\begin{proof}[Proof of \cref{2tMain}]
  Let us first prove the last statement about the codimension of the
  wall arrangements with unlucky stacks of mirror points. To have such a stack,~$l$ walls are needed to satisfy the restrictions of
  \cref{Stack}. These restrictions leave~$3$ degrees of
  freedom for the mirror points of these~$l$ walls, which gives an
  affine subspace of codimension $3 l - 3 = 3 (l-1)$. The wall
  arrangements with unlucky stacks are contained in the (finite) union of these subspaces
  associated to each choice of~$l$ walls, so we obtain the claimed codimension. Having thus shown the last
  statement, let us now turn to proving the main statement.
  
  A configuration of microphones on the vehicle is given by initial
  positions~$\ve m_i^\ini \in \RR^3$. For a matrix
  $A = \left(\begin{smallmatrix} a_{1,1} & a_{1,2} & a_{1,3} \\
      a_{2,1} & a_{2,2} & a_{2,3}\end{smallmatrix}\right) \in \RR^{2
    \times 3}$ such that
  $\left(\begin{smallmatrix} a_{1,1} & a_{1,2} \\ a_{2,1} &
      a_{2,2}\end{smallmatrix}\right) \in \RR^{2 \times 3}$ is
  orthogonal with determinant~$1$, consider the map
  \begin{equation} \label{eqPhi}%
    \mapl{\phi = \phi_A}{\RR^3}{\RR^3}{v}{
      \begin{pmatrix}
        a_{1,1} & a_{1,2} & 0 \\
        a_{2,1} & a_{2,2} & 0 \\
        0 & 0 & 1
      \end{pmatrix} \cdot v +
      \begin{pmatrix}
        a_{1,3} \\ a_{2,3} \\ 0
      \end{pmatrix}},
  \end{equation}
  and write $\ASO(2)$ for the group of all such maps. This group
  describes the possible vehicle positions. If a vehicle position is
  given by $\phi \in \ASO(2)$, then the microphone positions are
  $\phi(\ve m_i^\ini)$. The correspondence $A \leftrightarrow \phi_A$
  makes $\ASO(2)$ into an affine variety in $\RR^6$.

  Let $\mathcal W$ be the (finite) set of walls from our room. In this
  proof we identify the walls with the planes containing them. The
  mirror points are given by $\ve s = \re_W(\ve L)$, the reflection of
  the loudspeaker position at a wall $W \in \mathcal W$.  We call a
  group element $\phi \in \ASO(2)$ \df{really good} if the following
  holds: for four (not necessarily distinct) walls
  $W_1 \upto W_4 \in \mathcal W$ the relation
  \[
    f_D\bigl(\lVert\re_{W_1}(\ve L) - \phi(\ve m_1^\ini)\rVert^2 \upto
    \lVert\re_{W_4}(\ve L) - \phi(\ve m_4^\ini)\rVert^2\bigr) = 0
  \]
  is satisfied only if there is an $i \in \{1 \upto 4\}$ such that
  \[
    \Vert\re_{W_j}(\ve L) - \phi(\ve m_j^\ini)\rVert =
    \Vert\re_{W_i}(\ve L) - \phi(\ve m_j^\ini)\rVert \quad \text{for
      all} \quad j \in \{1 \upto 4\}.
  \]
  (The index~$i$ on the right hand side is not a typo, but the whole
  point of the condition.) It follows immediately that if~$\phi$ (or
  the vehicle position given by~$\phi$) is very good, then it is
  really good. More important is the following claim.
  \begin{claim} \label{2Claim1}%
    If $\phi \in \ASO(2)$ is really good, then the vehicle position
    given by~$\phi$ is good.
  \end{claim}
  Indeed, with the notation of \cref{1aDetect}, let
  $(d_1 \upto d_4) \in \mathcal D_1 \times \cdots \times \mathcal
  D_4$. For each~$j$ there exists a wall $W_j \in \mathcal W$ such
  that $d_j = \lVert\re_{W_j}(\ve L) - \phi(\ve m_j^\ini)\rVert^2$. If
  the algorithm detects a wall from the tuple $(d_1 \upto d_4)$ , then
  $f_D(d_1 \upto d_4) = 0$, so by hypothesis we have an~$i$ such that
  \begin{equation} \label{eqij}%
    \Vert\re_{W_i}(\ve L) - \phi(\ve m_j^\ini)\rVert^2 =
    \Vert\re_{W_j}(\ve L) - \phi(\ve m_j^\ini)\rVert^2 = d_j \quad
    \text{for all} \ j.
  \end{equation}
  Step~\ref{1aDetect6} of the algorithm computes a wall whose mirror
  point~$\ve s$ satisfies
  $\lVert\ve s - \phi(\ve m_j^\ini)\rVert^2 = d_j$ for all~$j$,
  so~\eqref{eqij} implies
  \[
    \bigl\langle\ve s - \re_{W_i}(\ve L),\phi(\ve
    m_j^\ini)\bigr\rangle = \frac{1}{2}\bigl(\lVert\ve s\rVert^2 -
    \lVert\re_{W_i}(\ve L)\rVert^2\bigr),
  \]
  where $\langle\cdot,\cdot\rangle$ denotes the standard scalar
  product.
  This holds for every $\phi(\ve m_j^\ini)$, so if $\ve s - \re_{W_i}(\ve L) \ne 0$ this would imply that the vectors $\phi(\ve m_j^\ini)$ lie on a common affine plane. But the~$\ve m_j^\ini$ and therefore also
  the~$\phi(\ve m_j^\ini)$ are not coplanar, so
  $\ve s = \re_{W_i}(\ve L)$. This means that the wall computed by the algorithm is
  $W_i$, a wall that actually exists. Therefore the algorithm
  detects no ghost walls.

  \begin{claim} \label{2Claim2}%
    Let $\ve s_1 \upto \ve s_4 \in \RR^3$ be vectors that do not form an unlucky stack. If
    \begin{equation} \label{eqClaim2h}%
      f_D\bigl(\lVert\ve s_1 - \phi(\ve m_1^\ini)\rVert^2 \upto
      \lVert\ve s_4 - \phi(\ve m_4^\ini)\rVert^2\bigr) = 0 \quad
      \text{for all} \ \phi \in \ASO(2),
    \end{equation}
    then there exists an $i \in \{1 \upto 4\}$ such that
    \begin{equation} \label{eqClaim2a}%
      \lVert\ve s_j - \phi(\ve m_j^\ini)\rVert = \lVert\ve s_i -
      \phi(\ve m_j^\ini)\rVert \quad \text{for all} \ \phi \in \ASO(2)
      \ \text{and all} \ j \in \{1 \upto 4\}.
    \end{equation}
  \end{claim}
  Before proving the claim, we show that it implies the theorem. For
  $W_1 \upto W_4 \in \mathcal W$, consider the set
  \[
    \mathcal U_{W_1 \upto W_4} := \bigl\{\phi \in \ASO(2) \mid
    f_D\bigl(\lVert\re_{W_1}(\ve L) - \phi(\ve m_1^\ini)\rVert^2 \upto
    \lVert\re_{W_4}(\ve L) - \phi(\ve m_4^\ini)\rVert^2\bigr) \ne
    0\bigr\}.
  \]
  Since
  $f_D\bigl(\lVert\re_{W_1}(\ve L) - \phi(\ve m_1^\ini)\rVert^2 \upto
  \lVert\re_{W_4}(\ve L) - \phi(\ve m_4^\ini)\rVert^2\bigr)$ depends
  polynomially on the coefficients of the matrix $A$ defining~$\phi$,
  and since $\ASO(2)$ is an irreducible variety, it follows that if
  $\mathcal U_{W_1 \upto W_4} \ne \emptyset$, then its complement in
  $\ASO(2)$ has dimension strictly less than $\dim(\ASO(2)) =
  3$. Therefore the intersection
  \[
    \mathcal U := \bigcap_{\substack{W_1 \upto W_4 \in \mathcal W\
        \text{such that} \\ \mathcal U_{W_1 \upto W_4} \ne \emptyset}}
    \mathcal U_{W_1 \upto W_4}
  \]
  has a complement of dimension~$\le 2$. To prove the theorem, it
  suffices to show that all $\phi \in \mathcal U$ are good under the
  hypothesis that there is no unlucky stack of mirror points. So let $\phi \in \mathcal U$. By \cref{2Claim1} it is enough
  to show that~$\phi$ is really good, so let
  $W_1 \upto W_4 \in \mathcal W$ such that
  $f_D\bigl(\lVert\re_{W_1}(\ve L) - \phi(\ve m_1^\ini)\rVert^2 \upto
  \lVert\re_{W_4}(\ve L) - \phi(\ve m_4^\ini)\rVert^2\bigr) =
  0$. Since $\phi \in \mathcal U$, this implies
  $\mathcal U_{W_1 \upto W_4} = \emptyset$. Now we apply
  \cref{2Claim2} (which we are assuming to be true), to
  $\ve s_i := \re_{W_i}(\ve L)$. This tells us that there is
  an~$i$ such that
  $\lVert\ve s_j - \phi(\ve m_j^\ini)\rVert = \lVert\ve s_i - \phi(\ve
  m_j^\ini)\rVert$ holds for all~$j$. (The claim makes this assertion
  for all~$\phi' \in \ASO(2)$, but we only need it for
  $\phi' = \phi$.)  But this is just what it means for $\phi$ to be
  really good.

  So we are left with proving \cref{2Claim2}. Thus we are given four
  non-coplanar vectors~$\ve m_i^\ini$ and four vectors
  $\ve s_1,\ve s_2,\ve s_3,\ve s_4 \in \RR^3$ that do not form an unlucky stack. To make the computations in the final part of
  the proof feasible, we ``preprocess'' the given data. For this, we
  use the group $G$ of all maps
  \[
    \mapl{\psi = \psi_{\sigma,\alpha,v_0}}{\RR^3}{\RR^3}{v}{\alpha
      \cdot \sigma(v) + v_0}
  \]
  where~$\sigma$ is an orthogonal map sending the $x$-$y$-plane to
  itself, $0 \ne \alpha \in \RR$, and $v_0 \in \RR^3$.  It is crucial
  but straightforward to check that $G$ normalizes $\ASO(2)$, i.e.,
  $\psi^{-1} \phi \psi \in \ASO(2)$ for $\psi \in G$ and
  $\phi \in \ASO(2)$.

  Now assume that we have chosen a suitable
  $\psi = \psi_{\sigma,\alpha,v_0}\in G$ and a $\phi_0 \in \ASO(2)$
  such that we can prove \cref{2Claim2} for
  $\tilde{\ve s}_i := \psi(\ve s_i)$ and
  $\tilde{\ve m}_i^\ini := (\phi_0 \circ \psi)(\ve m_i^\ini)$. To show
  that \cref{2Claim2} then also follows for the original~$\ve s_i$
  and~$\ve m_i^\ini$, we first verify that the $\tilde{\ve s}_i$ do not form an unlucky stack. First, if two of the $\ve s_i$ have
  different~$x$- or $y$-coordinates, then the same is true for
  the~$\tilde{\ve s}_i$. Moreover, the $z$-coordinate of
  $\tilde{\ve s}_i - \tilde{\ve s}_j$ is~$\pm \alpha$ times the
  $z$-coordinate of $\ve s_i - \ve s_j$; and, likewise, the
  $z$-coordinate of $\tilde{\ve m}_i^\ini - \tilde{\ve m}_j^\ini$
  is~$\pm \alpha$ times the $z$-coordinate of
  $\ve m_i^\ini - \ve m_j^\ini$. It follows that indeed the
  $\tilde{\ve s}_i$ do not form an unlucky stack.

  Now assume that
  $f_D\bigl(\lVert\ve s_1 - \phi(\ve m_1^\ini)\rVert^2 \upto \lVert\ve
  s_4 - \phi(\ve m_4^\ini)\rVert^2\bigr) = 0$ for all
  $\phi \in \ASO(2)$. We need to deduce the
  assertion~\eqref{eqClaim2a} of \cref{2Claim2} from this. Let
  $\phi' \in \ASO(2)$. Then
  $\phi := \psi^{-1} \phi' \phi_0 \psi \in \ASO(2)$ and
  $\phi'(\tilde{\ve m}_i^\ini) = (\psi \circ \phi)(\ve m_i^\ini)$. For
  $i,j \in \{1 \upto 4\}$ we obtain
  \begin{equation} \label{eqdij}%
    \lVert\tilde{\ve s}_i - \phi'(\tilde{\ve m}_j^\ini)\rVert^2 =
    \left\lVert \psi\bigl(\ve s_i - \phi(\ve
      m_j^\ini)\bigr)\right\rVert^2 = \alpha^2
    \left\lVert\sigma\bigl(\ve s_i - \phi(\ve
      m_j^\ini)\bigr)\right\rVert^2 = \alpha^2 \lVert\ve s_i -
    \phi(\ve m_j^\ini)\rVert^2.
  \end{equation}
  For the
  $D_{i,j} = \lVert\phi(\ve m_i^\ini) - \phi(\ve m_j^\ini)\rVert^2$
  that go into the polynomial~$f_D$, the same calculation shows
  $\tilde{D}_{i,j} := \lVert\phi'(\tilde{\ve m}_i^\ini) -
  \phi'(\tilde{\ve m}_j^\ini)\rVert^2 = \alpha^2
  D_{i,j}$. With~\eqref{eqdij} we obtain
  \begin{multline*}
    f_{\tilde{D}}\bigl(\lVert\tilde{\ve s}_1 - \phi'(\tilde{\ve
      m}_1^\ini)\rVert^2 \upto \lVert\tilde{\ve s}_4 -
    \phi'(\tilde{\ve m}_4^\ini)\rVert^2\bigr) = \det
    \begin{pmatrix}
      0 & \tilde{d}_1 & \cdots & \tilde{d}_4 & 1 \\
      \tilde{d}_1 & \tilde{D}_{1,1} & \cdots & \tilde{D}_{1,4} & 1 \\
      \vdots & \vdots & & \vdots & \vdots \\
      \tilde{d}_4 & \tilde{D}_{4,1} & \cdots & \tilde{D}_{4,4} & 1 \\
      1 & 1 & \cdots & 1 & 0
    \end{pmatrix} = \\
    \alpha^8 \det
    \begin{pmatrix}
      0 & d_1 & \cdots & d_4 & 1 \\
      d_1 & D_{1,1} & \cdots & D_{1,4} & 1 \\
      \vdots & \vdots & & \vdots & \vdots \\
      d_4 & D_{4,1} & \cdots & D_{4,4} & 1 \\
      1 & 1 & \cdots & 1 & 0
    \end{pmatrix} = \alpha^8 f_D\bigl(\lVert\ve s_1 - \phi(\ve
    m_1^\ini)\rVert^2 \upto \lVert\ve s_4 - \phi(\ve
    m_4^\ini)\rVert^2\bigr) = 0.
  \end{multline*}
  Since we are assuming \cref{2Claim2} for the $\tilde{\ve s}_i$ and
  $\tilde{\ve m}_i^\ini$, there is an $i \in \{1 \upto 4\}$ such that
  $\lVert\tilde{\ve s}_j - \phi'(\tilde{\ve m}_j^\ini)\rVert =
  \lVert\tilde{\ve s}_i - \phi'(\tilde{\ve m}_j^\ini)\rVert$ for
  all~$j$ and all $\phi' \in \ASO(2)$.

  Now let $\phi \in \ASO(2)$ and set
  $\phi' := \psi \phi \psi^{-1} \phi_0^{-1} \in
  \ASO(2)$. Then~\eqref{eqdij} yields
  \[
    \lVert\ve s_j - \phi(\ve m_j^\ini)\rVert = |\alpha|^{-1}
    \lVert\tilde{\ve s}_j - \phi'(\tilde{\ve m}_j^\ini)\rVert =
    |\alpha|^{-1} \lVert\tilde{\ve s}_i - \phi'(\tilde{\ve
      m}_j^\ini)\rVert = \lVert\ve s_i - \phi(\ve m_j^\ini)\rVert
  \]
  for all~$j$, so indeed \cref{2Claim2} follows for the $\ve s_i$ and
  $\ve m_i^\ini$ if it is true for the
  $\tilde{\ve s}_i = \psi(\ve s_i)$ and
  $\tilde{\ve m}_i^\ini = (\phi_0 \circ \psi)(\ve m_i^\ini)$.

  We use this to simplify the~$\ve s_i$ and~$\ve m_i^\ini$ in five
  steps.
  First, writing
  $\ve s_i = \left(\begin{smallmatrix} s_{i,1} \\ s_{i,2} \\
      s_{i,3}\end{smallmatrix}\right)$ and $\ve m_i^\ini =
  \left(\begin{smallmatrix} m_{i,1} \\ m_{i,2} \\
      m_{i,3}\end{smallmatrix}\right)$, we use
  $\psi = \psi_{\id,1,v_0} \in G$ with
  $v_0 = \left(\begin{smallmatrix} - s_{1,1} \\ - s_{1,2} \\ -
      m_{1,3}\end{smallmatrix}\right)$. Then $\tilde{\ve s}_1$
  has~$x$- and $y$-coordinates ~$0$, and $\tilde{\ve m}_1^\ini$ has
  $z$-coordinate~$0$. Hence without loss we may assume these things of
  $\ve s_1$ and $\ve m_1^\ini$. The second step comes from a
  QR-decomposition if the top two rows of
  $S := (s_{i,j}) \in \RR^{3 \times 4}$. We have
  \[
    \begin{pmatrix}
      s_{2,1} & s_{3,1} & s_{4,1} \\
      s_{2,2} & s_{3,2} & s_{4,2}
    \end{pmatrix}
    = Q R
  \]
  with $Q \in \Or(2)$ orthognal and $R \in \RR^{2 \times 3}$ upper
  triangular. Now forming~$\sigma \in \Or(3)$ with $Q^{-1}$ as upper
  left part, and using $\psi_{\sigma,1,0}$, we may assume
  $s_{2,2} = 0$. A bit more can be done: if the upper
  $2 \times 4$-part of $S$ is nonzero and $k \in \{2,3,4\}$ is the
  number of the first nonzero column, then we may assume
  $s_{k,1} \ne 0$ and $s_{k,2} = 0$. As a third step, we use
  $\psi_{\id,\alpha,0}$ with $\alpha = s_{k,1}^{-1}$. This means we
  can additionally assume $s_{k,1} = 1$. Summing up the first three
  steps, we may assume
  \begin{equation} \label{eqS}%
    S = \begin{pmatrix}
      \ve s_1 & \ve s_2 & \ve s_3 & \ve s_4
    \end{pmatrix} =
    \begin{pmatrix}
      0 & 1 & b_1 & b_2 \\
      0 & 0 & b_3 & b_4 \\
      b_5 & b_6 & b_7 & b_8
    \end{pmatrix}
  \end{equation}
  with $b_i \in \RR$. (This is for the case~$k = 2$; for $k = 3$
  or~$4$, the upper part has more zeroes, and as a last case, which we
  indicate by setting $k := 5$, it may be all zeroes.) For the
  $\ve m_i^\ini$ we have only achieved $m_{1,3} = 0$, but so far we
  have only used transformations where $\phi_0 = \id$. Now, as the
  fourth step, we set $\phi_0 \in \ASO(2)$ to be the translation by
  the vector~$- \ve m_1$. This yields $\tilde{\ve m}_1^\ini = 0$, so
  we may assume $\ve m_1^\ini = 0$. The last step uses a
  QR-decomposition of the upper $2 \times 4$-part of
  $M_\ini := (m_{i,j}) \in \RR^{3 \times 4}$ as we did before for
  $S$. (Notice that $Q$ can be assumed special orthogonal.) So finally
  we may assume
  \begin{equation} \label{eqM}%
    M_\ini=
    \begin{pmatrix}
      \ve m^\ini_1 & \ve m^\ini_2 & \ve m^\ini_3 & \ve m^\ini_4
    \end{pmatrix} =
    \begin{pmatrix}
      0 & c_1 & c_2 & c_3 \\
      0 & 0 & c_4 & c_5 \\
      0 & c_6 & c_7 & c_8
    \end{pmatrix}    
  \end{equation}
  with $c_i \in \RR$. The hypothesis that the~$\ve m_i^\ini$ are not
  coplanar translates into $\det\left(\begin{smallmatrix} c_1 & c_2 &
      c_3 \\ 0 & c_4 & c_5 \\ c_6 & c_7 & c_8\end{smallmatrix}\right)
  \ne 0$.

  For $A = \left(\begin{smallmatrix} a_{1,1} & a_{1,2} & a_{1,3} \\
      a_{2,1} & a_{2,2} & a_{2,3}\end{smallmatrix}\right) \in \RR^{2
    \times 3}$ with
  $\left(\begin{smallmatrix} a_{1,1} & a_{1,2} \\ a_{2,1} &
      a_{2,2}\end{smallmatrix}\right) \in \RR^{2 \times 3}$ orthogonal
  of determinant~$1$, and for $S$ and $M_\ini$ as in~\eqref{eqS}
  and~\eqref{eqM}, there are polynomials
  $F(x_{1,1} \upto x_{2,3},y_1 \upto y_8,z_1 \upto z_8)$ and
  $F_{i,j}(x_{1,1} \upto x_{2,3},y_1 \upto y_8,z_1 \upto z_8)$ in~$22$
  indeterminates such that
  \[
    \lVert\ve s_i - \phi_A(\ve m_j^\ini)\rVert^2 = F_{i,j}(a_{1,1}
    \upto a_{2,3},b_1 \upto b_8,c_1 \upto c_8) \quad (i,j \in \{1
    \upto 4\})
  \]
  and
  \[
    f_D\bigl(\lVert\ve s_1 - \phi_A(\ve m_1^\ini)\rVert^2 \upto
    \lVert\ve s_4 - \phi_A(\ve m_4^\ini)\rVert^2\bigr) = F(a_{1,1}
    \upto a_{2,3},b_1 \upto b_8,c_1 \upto c_8).
  \]
  (In the case where $k \ge 3$, there are fewer~$b_i$ and hence
  fewer~$y_i$.)  Let $I \subset \RR[x_{1,1} \upto x_{2,3}]$ be the
  ideal generated by~$x_{1,1} - x_{2,2}$, $x_{1,2} + x_{2,1}$
  and~$x_{2,1}^2 + x_{2,2}^2$.  This is the vanishing ideal of
  $\ASO(2)$ as a subvariety of $\RR^6$. Choose a Gr\"obner basis $G$
  of $I$ with respect to an arbitrary monomial ordering, such as the
  ideal basis given above, and consider the normal forms
  $\tilde{F} := \NF_G(F)$ and $\tilde{F}_{i,j} := \NF_G(F_{i,j})$. Let
  $J \subseteq \RR[y_1 \upto y_8,z_1 \upto z_8]$ be the ideal
  generated by the coefficients of $\tilde{F}$, viewed as a polynomial
  in $x_{1,1} \upto x_{2,3}$ with coefficients in
  $\RR[y_1 \upto y_8,z_1 \upto z_8]$. Moreover, for $i = 1 \upto 4$,
  let $D_i \subseteq \RR[y_1 \upto y_8,z_1 \upto z_8]$ be the ideals
  generated by the coefficients of all
  $\tilde{F}_{i,j} - \tilde{F}_{j,j}$ ($j = 1 \upto 4$).  Also set $d
  := \det\left(\begin{smallmatrix} z_1 & z_2 & z_3 \\ 0 & z_4 & z_5 \\
      z_6 & z_7 & z_8\end{smallmatrix}\right)$.
  \begin{claim} \label{2Claim3}%
    If there is an~$r$ such that for each $k \in \{2 \upto 5\}$
    \begin{equation} \label{eqClaim3a}%
      \bigl((d) \cdot D_1 \cdots D_4\bigr)^r \subseteq J \quad
      \text{in the case} \ k \le 4
    \end{equation}
    or
    \begin{equation} \label{eqClaim3b}%
      \bigl((y_6 - y_5 - 2 z_6,y_7 - y_5 - 2 z_7,y_8 - y_5 - 2
      z_8)\cdot (d) \cdot D_1 \cdots D_4\bigr)^r \subseteq J \quad
      \text{in the case} \ k = 5,
    \end{equation}
    then \cref{2Claim2} and therefore the theorem follow. Recall
    that~$k$ is the number of the first nonzero column of the upper
    $2 \times 4$-submatrix of $S$, with $k = 5$ indicating that this
    submatrix is zero.
  \end{claim}
  In fact, we have already seen that to prove \cref{2Claim2}, we may
  assume the~$\ve s_i$ and~$\ve m_i^\ini$ to be given by~\eqref{eqS}
  and~\eqref{eqM}. Assume that the assertion~\eqref{eqClaim2a} of
  \cref{2Claim2} is not true, so for every~$i = 1 \upto 4$ there is a
  $j \in \{1 \upto 4\}$ and a $\phi = \phi_A \in \ASO(2)$ such that
  $\lVert\ve s_j - \phi_A(\ve m_j^\ini)\rVert \ne \lVert\ve s_i -
  \phi_A(\ve m_j^\ini)\rVert$. Then
  \[
    0 \ne F_{i,j}(a_{1,1} \upto a_{2,3},b_1 \upto b_8,c_1 \upto c_8) -
    F_{j,j}(a_{1,1} \upto a_{2,3},b_1 \upto b_8,c_1 \upto c_8).
  \]
  Since $F_{i,j} - \tilde{F}_{i,j}$ vanishes at $x_{i,j} = a_{i,j}$,
  we may replace $F_{i,j}- F_{j,j}$ in the above inequality by
  $\tilde{F}_{i,j}- \tilde{F}_{j,j}$. So there is a~$j$ such that some
  coefficient of $\tilde{F}_{i,j}- \tilde{F}_{j,j}$, viewed as a
  polynomial in $x_{1,1} \upto x_{2,3}$, does not vanish when
  evaluated at $y_i = b_i$ and $z_i = c_i$. This coefficient is one of
  the generators of $D_i$, so each $D_i$ contains an element that does
  not vanish at $y_i = b_i$ and $z_i = c_i$.

  Moreover, $d$ does not vanish when evaluated at $z_i = c_i$, and we
  might as well also evaluate it at~$y_i = b_i$ since the~$y_i$ do not
  occur in~$d$. So in the case $k \le 4$, \eqref{eqClaim3a} implies
  that there exists an element of $J$ that does not vanish when
  evaluated at~$y_i = b_i$ and~$z_i = c_i$.

  On the other hand, if~$k = 5$, then all~$\ve s_i$ share the
  same~$x$- and~$y$-coordinates, so the hypothesis that they do not form an unlucky stack translates into $b_i - b_5 \ne 2 c_i$ for
  some $i \in \{6,7,8\}$. So one of the generators of the first ideal
  in the product in~\eqref{eqClaim3b} does not vanish when evaluated
  at $y_i = b_i$. In this case~\eqref{eqClaim3b} yields the same
  conclusion: that $J$ contains an element that does not vanish when
  evaluated at~$y_i = b_i$ and~$z_i = c_i$.

  Because of the way $J$ was constructed, this means that at least one
  coefficient of $\tilde{F}$ does not vanish at $y_i = b_i$ and
  $z_i = c_i$, so
  \[
    \tilde{F}(x_{1,1} \upto x_{2,3},b_1 \upto b_8,c_1 \upto c_8) \ne
    0.
  \]
  The linearity of the normal form map implies
  $\NF_G\bigl(F(x_{1,1} \upto x_{2,3},b_1 \upto b_8,c_1 \upto
  c_8)\bigr) = \tilde{F}(x_{1,1} \upto x_{2,3},b_1 \upto b_8,c_1 \upto
  c_8)$, so we conclude that
  $F(x_{1,1} \upto x_{2,3},b_1 \upto b_8,c_1 \upto c_8)$ has nonzero
  normal form and therefore does not lie in $I$. Since $I$ is the
  vanishing ideal of $\ASO(2)$, this means that there exists a
  $\phi = \phi_A \in \ASO(2)$ such that
  \[
    0 \ne F(a_{1,1} \upto a_{2,3},b_1 \upto b_8,c_1 \upto c_8) =
    f_D\bigl(\lVert\ve s_1 - \phi(\ve m_1^\ini)\rVert^2 \upto
    \lVert\ve s_4 - \phi(\ve m_4^\ini)\rVert^2\bigr).
  \]
  So the hypothesis~\eqref{eqClaim2h} of \cref{2Claim2} is not true
  under the assumption that the assertion is not true. So indeed
  \cref{2Claim2} follows from~\eqref{eqClaim3a} and~\eqref{eqClaim3b}.

  It remains to verify~\cref{eqClaim3a} and~\eqref{eqClaim3b}, and this
  can be checked with the help of a computer. For the computation, we
  used MAGMA~[\citenumber{magma}] and proceeded as
  follows: 
  \begin{itemize}
  \item We ran the following steps for each~$k \in \{2 \upto 5\}$. We
    will not introduce any notation to indicate the cases for
    different~$k$ below.
  \item It is straightforward to compute the polynomials $F$,
    $F_{i,j}$, $\tilde{F}$, and $\tilde{F}_{i,j}$ according to their
    definitions, and to pick out the sets of coefficients
    $C,C_{i,j} \subset \RR[y_1 \upto y_8,z_1 \upto z_8]$ of
    $\tilde{F}$ and the $\tilde{F}_{i,j}$.
  \item Using an additional indeterminate~$t$, we computed the set
    $C^\h$ of homogenizations of the polynomials in $C$ with respect
    to~$t$.
  \item We computed truncated Gr\"obner bases $G^\h$ of the ideal
    generated by $C^\h$ of rising degree.
  \item For each degree we computed the normal form
    $\NF_{G^\h}\bigl(t^i d^j\bigr)$ for all~$i$ and~$j$ such that
    $t^i d^j$ has the degree up to which the Gr\"obner basis was
    computed. If at least one of the normal forms is zero, this shows
    that $t^i d^j$ is an $\RR[y_1 \upto y_8,z_1 \upto z_8,t]$-linear
    combination of the polynomials in $C^\h$. Setting~$t = 1$ shows
    that $d^j \in J$, so~\cref{eqClaim3a} holds with $r = j$. As it
    turned out, this was successful for all $k \le 4$, and we never
    needed to go beyond degree~$11$. This means that in these cases
    the ideals $D_i$ are not actually needed in~\cref{eqClaim3a}.
  \item Finally we dealt with the case $k = 5$. Here, thanks to the
    reduced number of veriables, we were able to directly compute the
    ideal $J_0 := J:(d)^\infty$, and then $J_i := J_{i-1}:D_i^\infty$
    for $i = 1 \upto 4$. The result is
    $J_4 = (y_6 - y_5 - 2 z_6,y_7 - y_5 - 2 z_7,y_8 - y_5 - 2 z_8)$,
    which shows~\eqref{eqClaim3b}.
  \end{itemize}
  The total computation time less than two minutes.
\end{proof}

\begin{theorem}[A hovering drone and a fixed
  loudspeaker] \label{2tHovering}%
  Assume the same situation as in \cref{2tMain}, except that the
  vehicle can take positions in the same way as a hovering drone. Then
  the assertions from \cref{2tMain} hold.
\end{theorem}

\begin{proof}
  The proof is exactly like the one of \cref{2tMain}. The only
  difference is that instead of $\ASO(2)$ we need to consider the
  group of all~$\phi$ as in~\eqref{eqPhi}, but with the last vector
  having a third component~$a_{3,3}$ instead of zero. The computer
  computations have to be run with modified input, and take about four
  minutes to finish. Optimizing the ``preprocessing'' of the data for
  the group used in this proof would probably shorten the computation
  time.
\end{proof}

\section{The two-dimensional case} \label{sDim2}

In this section we briefly consider the purely two-dimensional case: a vehicle moves in the plane, and all reflecting walls are also in this plane. A physical example of such a scene is a robot navigating in a warehouse, with no echoes coming back from either the ceiling or the floor, nor from any inclined walls. In this case only three microphones on the vehicle are required. The loudspeaker can be at a fixed position, but will more typically be mounted on the vehicle. In contrast to the case of a ground-based vehicle moving in 3D, we can also handle the case of a mounted loudspeaker here. In fact, this is the two-dimensional variant of the situation considered in~[\citenumber{Boutin:Kemper:2019}], and everything in that paper carries over directly to other dimensions. This includes the relation, given as~\eqref{1eqRel} in this paper, and the wall detection algorithm. Moreover, the MAGMA-programs used for the computational verifications in~[\citenumber{Boutin:Kemper:2019}] were written for general dimension. As it turn out, if the dimension is set to~$2$, the entire computations require only about one second. (They take about five minutes for the three-dimensional case considered in~[\citenumber{Boutin:Kemper:2019}].) Notice that for a two-dimensional modelling to be admissible, the microphones and the loudspeaker need to be contained in a common plane, which is parallel to the plane of motion.
The following result emerges.

\begin{theorem}[A vehicle in a two-dimensional scene] \label{3t2D}%
  Assume a vehicle can move in a two-dimensional scene, which contains a finite number of walls. Assume a loudspeaker is either placed at a fixed position in the scene or mounted on the vehicle, and that three microphones are mounted on the vehicle, but do not lie on a common line. In any case, the microphones and the loudspeaker need to be in a common plane, which is parallel to the plane of motion. Then almost all vehicle positions are good, again in the sense that all walls whose echoes are heard by every microphone are detected, but no ghost walls are detected.
\end{theorem}

In the two-dimensional situation the issue of unlucky stacks does not occur.

\bibliographystyle{mybibstyle} \bibliography{bib}

\end{document}